\documentclass[12pt]{article}
\usepackage{fullpage}
\usepackage{latexsym}
\usepackage{epsfig}
\usepackage{rotating}
\usepackage{amssymb}
\usepackage[T1]{fontenc}
\usepackage{afterpage}

\usepackage{amssymb}
\usepackage{amsmath}
\usepackage[all]{xy}
\usepackage{amsmath}
\usepackage{amsfonts}
\usepackage{amssymb}
\usepackage{amscd}
\usepackage{amsthm}
\usepackage{latexsym}
\usepackage{amsbsy}
\usepackage{graphicx}
\usepackage{float}
\usepackage[toc,page]{appendix}
\usepackage{caption}
\usepackage{subcaption}
\usepackage[a4paper]{geometry}
\usepackage{float}

\makeatletter
\theoremstyle{definition}

\newtheorem{theorem}{Theorem}[section]
\newtheorem{definition}{Definition}[section]

\newtheorem{prop}{Proposition}

\newtheorem*{remark}{Remark}
\DeclareMathOperator{\tr}{Tr}

\title{Enumerating planar stuffed maps as hypertrees of mobiles} 
\author{
\parbox{8cm}{
\centerline {\sc Nathan Pagliaroli}
\centerline{\small University of Waterloo}
\centerline{\small\tt npagliar@waterloo.ca }}}

\usepackage[backend=bibtex]{biblatex}
\begin{document}
	
	\maketitle
	
	\begin{abstract}
		 A planar stuffed map is an embedding of a graph into the 2-sphere $S^{2}$, considered up to orientation-preserving homeomorphisms, such that the complement of the graph is a collection of disjoint topologically connected components that are each homeomorphic to $S^{2}$ with multiple boundaries. This is a generalization of planar maps whose complement of the graph is a collection of disjoint topologically connected components that are each homeomorphic to a disc. The main goal of this work is to construct a bijection between bipartite planar stuffed maps and collections of integer-labelled trees connected by hyperedges such that they form a hypertree, called hypermobiles. This bijection directly generalizes the Bouttier-Di Franceso-Guitter bijection between bipartite planar maps and mobiles. Additionally, we show that the generating functions of these trees of mobiles satisfy both an algebraic equation, generalizing the case of ordinary planar maps, and a new functional equation.  As an example, we explicitly enumerate a class of stuffed quadrangulations.
	\end{abstract}
	
	
	\section{Introduction}

  Planar maps were first studied and enumerated by Tutte in the 1960's \cite{tutte1968enumeration} and have since appeared in statistical mechanics, quantum gravity, and probability theory. For a detailed history see \cite{schaeffer2015planar}. One key insight discovered was that planar maps are in bijection with decorated trees \cite{cori1981planar,schaeffer1998conjugaison} from which enumerative formulae could be derived. Of particular relevance to this work is the bijection of Bouttier-Di Franceso-Guitter (BDFG) \cite{bouttier2004planar}, which is between planar maps and plane trees with integer-labelled vertices called mobiles. Such bijections have since been generalized beyond planar maps \cite{chapuy2009bijection,miermont2009tessellations}.  In addition to enumeration, these bijections gave insight into the discrete geometry of maps which has found applications in theoretical physics and probability theory \cite{bouttier2003geodesic,le2010scaling}. 


    In this work we construct a bijection between stuffed maps and types of decorated hypertrees which generalizes the BDFG bijection for bipartite stuffed maps. Instead of mobiles, the bijection constructed is for what we refer to as hypermobiles. Hypermobiles are collections of integer labelled plane trees in which some of the vertices of the mobiles are connected by hyperedges such that the underlying hypergraph is a hypertree. When restricted to ordinary maps, this bijection reduces to the BDFG bijection and hypermobiles to mobiles. The first half of the paper focuses on the construction and proof of this bijection. The second half focuses on planar stuffed map and hypermobile enumeration. In particular, we derive a generalization of the tree equation for certain generating functions as well as a functional equation for pointed stuffed maps.  We also give a detailed combinatorial interpretation of the bijection induced by the functional equation.  Lastly, as an example, we explicitly enumerate a type of stuffed quadrangle using the tree equation and functional equation. 

   As the reader will see, stuffed maps have many interesting combinatorial properties, and their enumeration is much more complicated than ordinary maps.  Stuffed maps were first introduced in \cite{borot2014formal}, where it was shown that stuffed maps of any genus with boundaries could be recursively enumerated using a process called blobbed topological recursion. This work was largely motivated by the study of multi-tracial matrix models. As models of quantum gravity, the multi-tracial terms in the potential of such matrix models contribute ``wormholes" in the two-dimensional universes of their Feynman diagrams i.e. stuffed maps. In particular, an interior of a 2-cell with multiple boundaries corresponds to a wormhole and a path-connected component corresponds to universes in the physics literature \cite{das1990new,ambjorn2001lorentzian}.  Such matrix models are also of interest in what is known as the duality of Liouville quantum gravity \cite{klebanov1995non,duplantier2009duality,rhodes2014gaussian}. Other applications include the $O(n)$ model which involves ordinary maps decorated with self-avoiding loop configurations which can be shown to be equivalent to stuffed maps as well as maps with intersecting loops.

We hope that this paper will lay the foundations for studying the geometric and statistical properties of stuffed maps. Stuffed maps have yet to be fully studied from the perspective of random maps. Recently, much work has been done to connect the Gromov-Hausdorf limit of random metric spaces to random surfaces from two-dimensional quantum gravity. Researchers are looking for new random surfaces, and many types of maps have been shown to converge to such spaces \cite{le2019brownian}. It is believed, based on critical exponent analysis, that stuffed maps will not fall into the same class as ordinary maps, and hence might be related to new types of random surfaces. One can enumerate stuffed maps by computing formal multi-tracial matrix integrals, and in the physics literature such integrals are conjectured to have a continuum limit with "baby-universes" \cite{korchemsky1992matrix,ambjorn1993baby}. Certain coloured stuffed maps appear also in the study of tensor models \cite{bonzom2017}. Establishing bijections with trees could be used as a first step in establishing a rigorous construction of a random surfaces with such features, as has been done for many types of ordinary maps \cite{le2010scaling, le2013uniqueness,bettinelli2017compact,gwynne2017scaling}. 

This article is organized as follows. Section 2 gives a brief review of stuffed maps and then proves the bijection with hypermobiles. Section 3 is dedicated to the enumeration of planar stuffed maps. An analogue of the tree-equation for maps is derived and a functional equation is established for pointed stuffed. A bijective interpretation of the latter is given.  In Section \ref{sec:unstab;e},  we explore the specific example of stuffed maps glued from quadrangles and 2-cells with two boundaries of length two.

 \subsection*{Acknowledgements} 
    The author would like to thank Jeremy Gamble for his helpful discussions and acknowledge the support of the Natural Sciences and Engineering Research Council of Canada (NSERC). 

	\section{Background on stuffed maps}
	\subsection{Planar stuffed maps}\label{sec: planar stuffed maps}
	A \textit{map} of genus $g$ is an equivalence class of 2-cell embeddings of a graph into an oriented surface of genus $g$, considered up to orientation-preserving homeomorphisms. By 2-cell embedding, we mean that the complement of the graph is a collection of disjoint topologically connected components that are each homeomorphic to an open disk, called a face. This concept also gives rise to a useful method for building maps. Since any map is a 2-cell embedding, by taking a collection of oriented polygons and gluing their edges, called \textit{ half-edges}, together in an orientation-preserving manner, one can construct any map. By an orientation-preserving manner, we mean that each face has a recto and verso side,  and one glues them such that all rectos are aligned in orientation and and all versos are aligned in the opposite orientation.

    The \textit{genus} of a connected map is the genus of the underlying surface. Throughout this work we restrict our attention solely to maps that are genus zero (planar) as well as \textit{bipartite,} i.e.  a map whose graph is bipartite. Bipartite maps can be equivalently characterized as only being glued from polygons with an even number of edges. We will also require some decorations of maps. A map is said to be \textit{pointed} if it has a distinguished vertex, which we refer to as the source. A map is \textit{rooted} if one of its 2-cells has a distinguished and oriented edge, denoted by a half-arrow.  A \textit{boundary} of a map is a distinguished and rooted face.

It is natural to consider maps with more general embeddings on surfaces. An \textit{(elementary planar) 2-cell} with $k$ perimeters of lengths $(\ell_{1},...,\ell_{k})$ is a topological oriented genus zero surface with $k$ polygonal boundaries of lengths $(\ell_{1},...,\ell_{k})$.   Any 2-cell with one boundary we will refer to as a polygon, otherwise we will refer to it as a \textit{multi-boundary 2-cell}. We will refer to the interior of a multi-boundary 2-cell as a \textit{branch}. A \textit{(planar) stuffed map} is an embedding of a graph into an oriented surface of genus $g$, considered up to orientation-preserving graph homomorphisms, such that the complement of the graph is a collection of disjoint topologically connected components that are each homeomorphic to elementary planar 2-cells. This paper restricts its attention to planar bipartite maps and planar bipartite stuffed maps, which we will simply refer to as (ordinary) maps and stuffed maps.  An orientation-preserving map homomorphism is a \textit{map homeomorphism}. For a given map, the automorphism group is a subgroup of the permutations of the possible labelling of the elementary 2-cells used to construct the map. As with ordinary maps, the genus of a stuffed map is the genus of the underlying surface.

\begin{figure}[h]\label{fig:example maps}
    \centering
    \begin{subfigure}[t]{0.5\textwidth}
        \centering
        \includegraphics[height=2.5in]{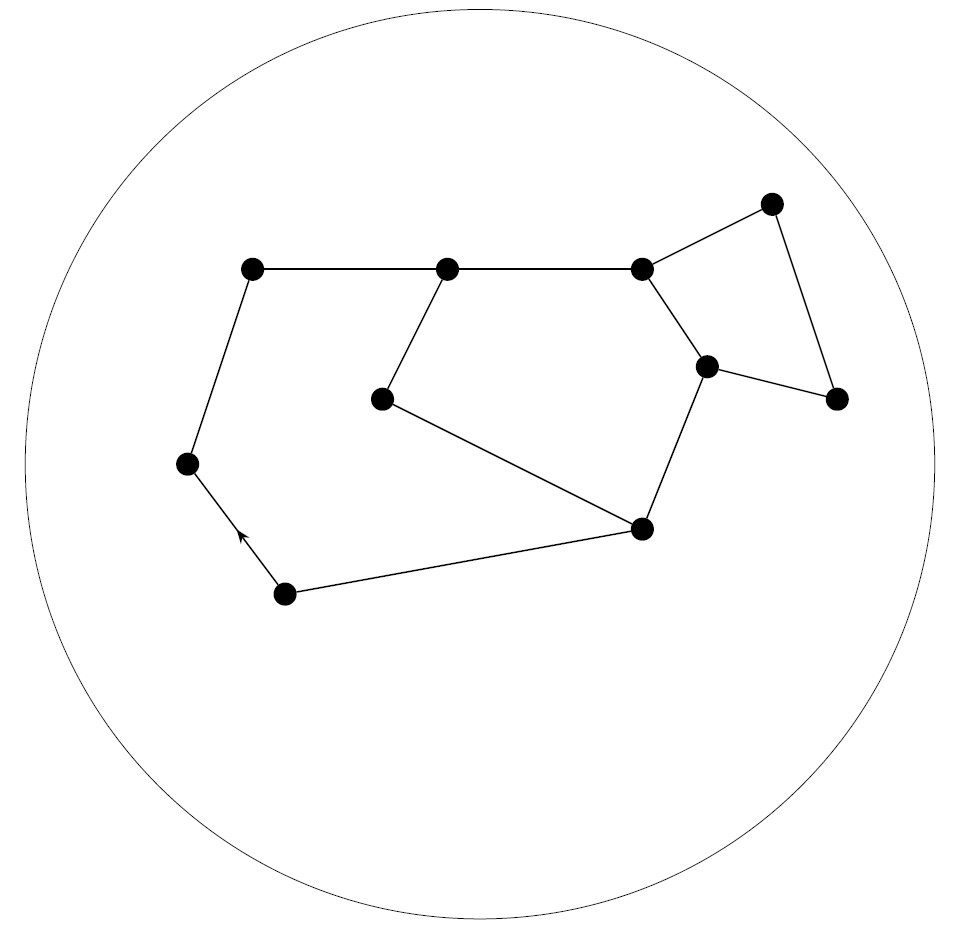}
        \caption{}
    \end{subfigure}%
    ~ 
    \begin{subfigure}[t]{0.5\textwidth} 
        \centering
        \includegraphics[height=2.5in]{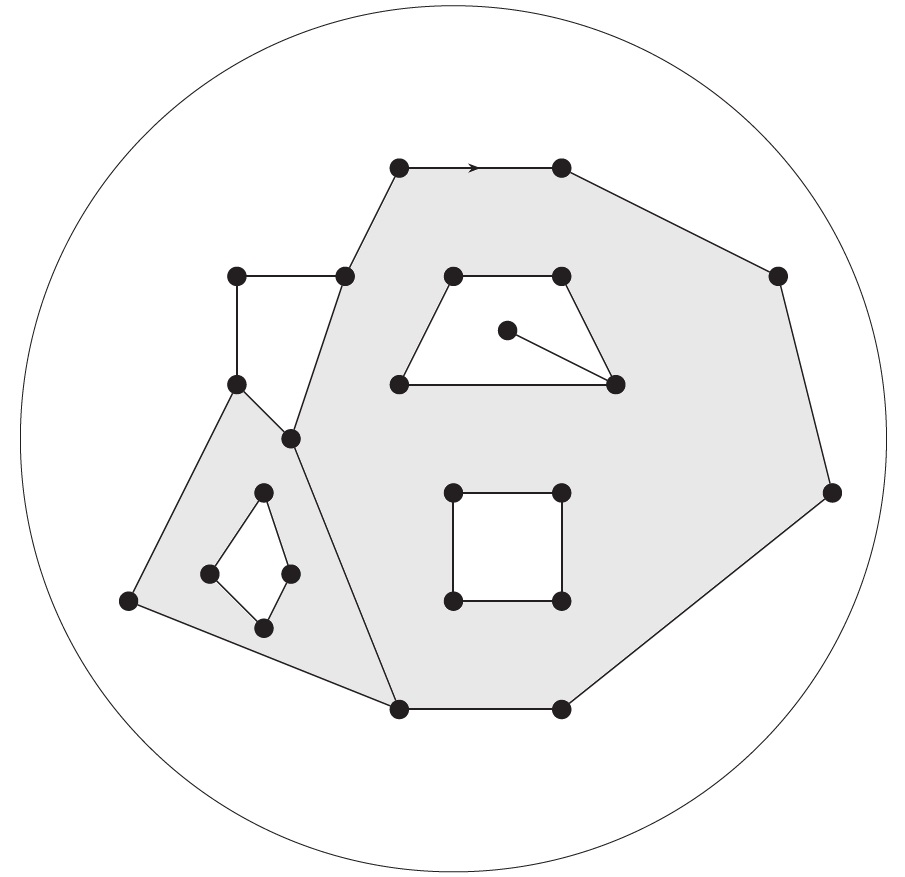}
        \caption{}
    \end{subfigure}
    \caption{Sub-figure (a) is an example of an bipartite planar rooted (ordinary) map. Sub-figure (b) is an example of a (bipartite planar) rooted stuffed map. The grey shaded region distinguishes a 2-cell with more than one boundary. In particular, this map is constructed from two such 2-cells: one with boundaries of lengths eight, six, and four, and the other with two boundaries both of length four.}
\end{figure}

\begin{definition}
    Given a set $S$ of 2-cells with even boundary lengths, for $v>1$ let $\mathbb{M}_{\ell}(S;v)$ denote the set of stuffed planar topologically connected maps with $v$ vertices glued in an orientation-preserving manner from a boundary of length $\ell$, and optionally elements of $S$. Let $d$ denote the largest boundary length of any 2-cell in $S$, which does not need to be finite. Let $\mathbb{M}(S;v)$ denote the same set without boundary but with a rooted edge. By convention, we define the only map in $\mathbb{M}_{1}(S;1)$ to be the trivial map, i.e. a single vertex, embedded into the sphere. 
\end{definition}

 One very important note about stuffed maps is that, unlike ordinary maps, a stuffed map being topologically connected does not imply that its underlying graph is connected.  Since the branches of a stuffed map do not contain vertices, edges, or faces, the Euler characteristic of a stuffed map is the same with or without its branches i.e. for a stuffed map $M$, it is 
 \begin{align*}
        \chi(M) = 2K(M) -2g(M)&=  V(M) - E(M) +F(M),
    \end{align*}
    where $K(M)$ is the number of connected components of $M$, $g(M)$ is the genus, $V(M)$ is the number of vertices, $E(M)$ the number of edges, and $F(M)$ the number of faces. For a rooted stuffed map, we will refer to the component with the root as the \textit{gasket}. As mentioned before, this paper is strictly interested in studying planar maps i.e. genus zero maps. 

To any stuffed map, one can canonically associate a hypergraph whose vertices are components and whose edges are branches. Every hyperedge is between connected components that are connected by the same branch. See Figure \ref{fig: associated hypertree} for an example. Note that this hypergraph is in fact a hypertree, since any cycle would correspond to a sequence of branches and components in the original graph that would violate planarity. We will refer to this hypergraph as the \textit{associated hypertree} $H_{M}$ to a stuffed map $M$.

\begin{figure}[H]
    \centering
        \centering
        \includegraphics[height=2.5in]{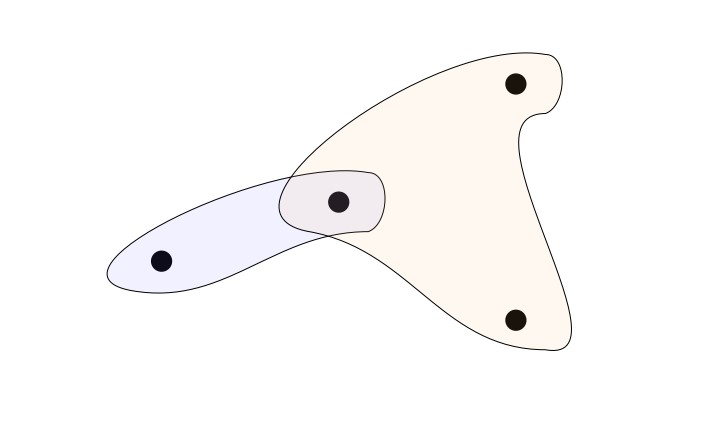}
    \caption{The associated hypertree to the stuffed map in Figure 1 (b). The middle vertex corresponds to the gasket.}
    \label{fig: associated hypertree}
\end{figure}

\subsection{Mobiles}
    Define a \textit{mobile} as a bipartite plane tree with black vertices and integer-labelled white vertices such that:
\begin{itemize}
    \item Any edge connects a white vertex to a black vertex.
    \item For any white vertex labelled $n$ that is attached to a black vertex, any adjacent clockwise white vertex attached to the same black vertex must be labelled at least $n-1$.
    \item If the mobile is rooted, the source of the rooted edge is a white vertex labelled zero.
\end{itemize}
 The BDFG bijeciton is between pointed rooted bipartite planar maps and rooted mobiles \cite{bouttier2004planar}, which generalized the bijection in \cite{schaeffer1998conjugaison} between quadrangulations and labelled trees. 

For the purpose of making this article self-contained, we will recall the BDFG bijection.  Let $M$ be a rooted and pointed (bipartite planar) map. Then we can construct the corresponding mobile $\Omega(M)$ via the following:
 \begin{enumerate}
     \item Set the colour of all vertices of $M$ to be white and draw a black vertex in every face of $M$.  Label the pointed vertex zero and then all other vertices with the graph distance from the pointed vertex.
     \item  For every edge $e$ of $M$, let $v$ be the vertex with the largest label. Draw an edge from the vertex $v$ to the black vertex in the face to the left $e$ when directed at $v$. If $e$ is the rooted edge, then take the new edge to be rooted, oriented away from $v$. 
         \item Erase all of the original edges, all vertices with label zero, and all branches of $M$. Uniformly shift all labels such that the source vertex of the root is labelled zero.
     \end{enumerate}

Given a vertex of a map, a \textit{corner} of said vertex refers to the angular region between two adjacent edges. Given a mobile $m$, we can construct a map $\Omega^{-1}(T)$ as follows:
\begin{enumerate}
 \item Add a white vertex labelled one less than the minimum label $n_{\text{min}}(m)$ of $m$.
        \item For every corner of a white vertex in $m$, if the label of the vertex is $n_{\text{min}}(m)$ then connect it to the white vertex added for that mobile in the previous step. Otherwise, if the label is $n \not = n_{\text{min}}(m)$, then connect it to the next white vertex labelled $n-1$ in a clockwise direction around the contour. If this corner is the corner to the left of a root, the new edge is rooted pointing away for the corner.
    \item Remove all the original edges and black vertices of $T$. 
\end{enumerate}

\subsection{Hypermobiles}
In this work, we required a generalization of mobiles. Recall that a hypergraph is a generalization of a graph where edges are replaced by hyperedges that can connect multiple vertices.  

\begin{figure}[H]
    \centering
    \begin{subfigure}[t]{0.5\textwidth}
        \centering
        \includegraphics[height=2.5in]{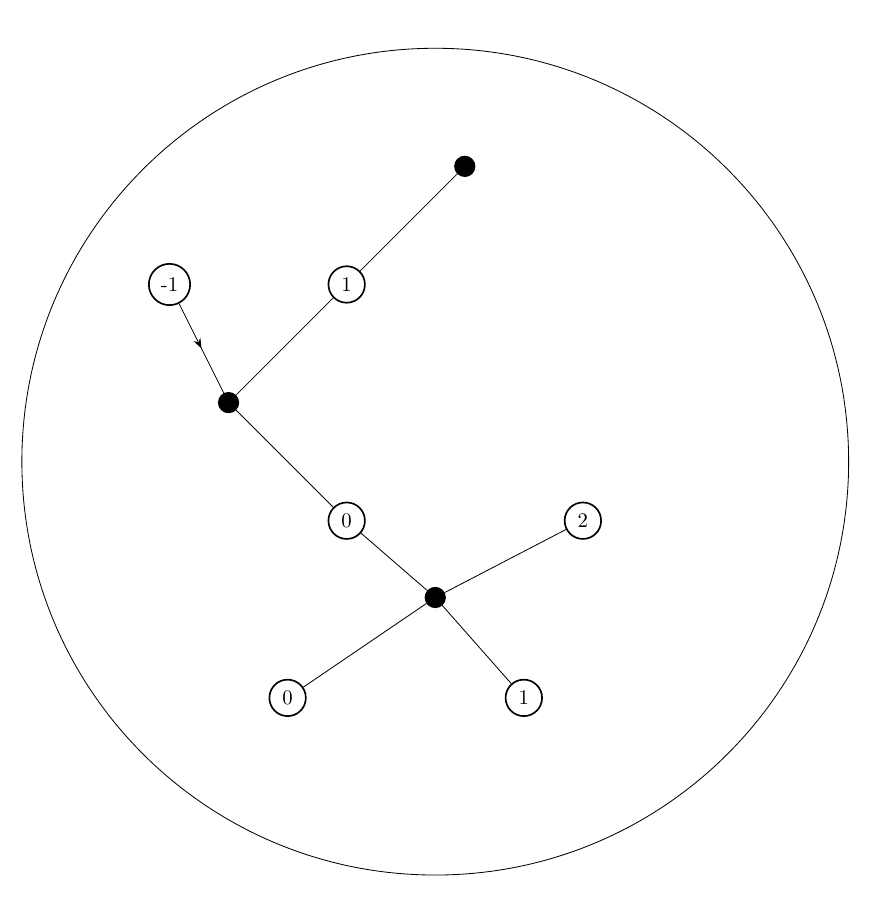}
        \caption{}
    \end{subfigure}%
    ~ 
    \begin{subfigure}[t]{0.5\textwidth}
        \centering
        \includegraphics[height=2.5in]{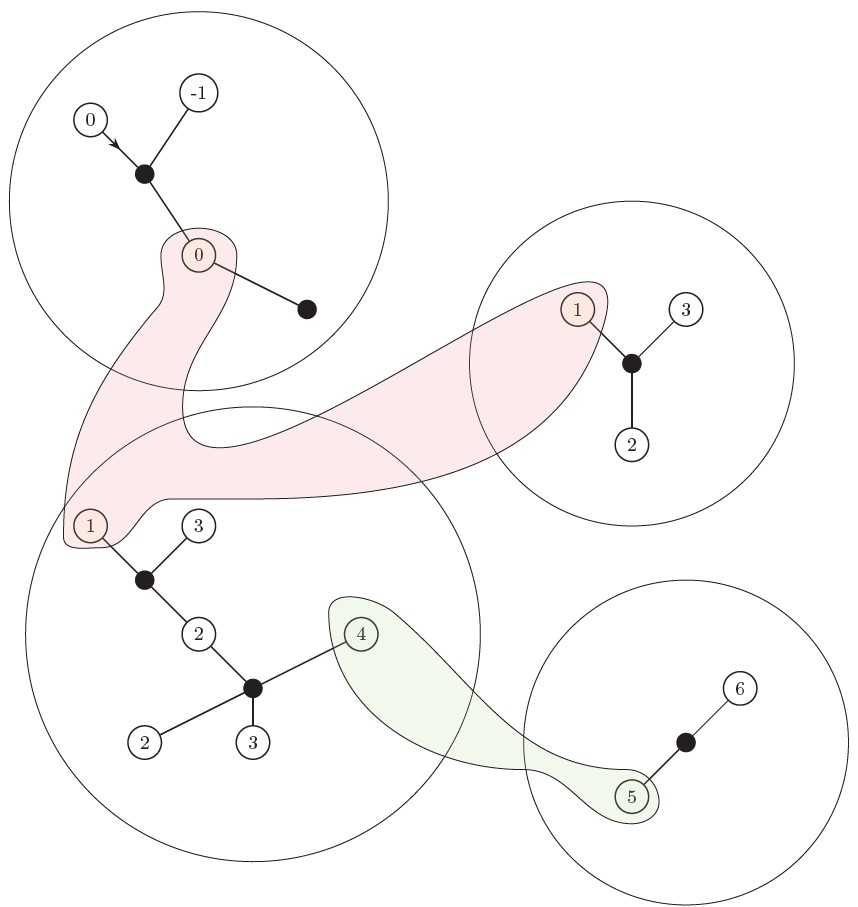}
        \caption{}
    \end{subfigure}
    \caption{Sub-figure (a) is an example of an mobile. Sub-figure (b) is an example of a hypermobile, where each hyperedge is represented by a different colour.}
\end{figure}
\begin{definition}
    We define a \textit{hypermobile} as a triple $(\mathcal{M},H,r)$ composed of:
\begin{itemize}
    \item A set of mobiles $\mathcal{M}$ with a unique  mobile with the lowest labelled vertex, called the \textit{origin mobile}. 
    \item A set of hyperedges $H$ between 
 the white vertices of distinct mobiles called \textit{gate vertices}, such that the underlying hypergraph is a hypertree. On every mobile that is not the origin mobile, the lowest labelled vertex is the gate vertex closest to the origin mobile. Its label is one more than the label of the gate vertex  closest to the origin mobile within that hyperedge.
 \item A root $r$ on an edge of a mobile whose source vertex is labelled zero, and called the \textit{origin}.
\end{itemize}
\end{definition}

To make our bijeciton between hypermobiles and stuffed maps more precise to every set $S$ of 2-cells we associate a set $S'$ of vertices and hyperedges as follows:
\begin{itemize}
    \item For each 2-cell with the topology of a disc whose boundary is length $2\ell$, one has a black vertex with valency $\ell$.
    \item For each 2-cell with $k>1$ boundaries of lengths $2\ell_{1},2 \ell_{2},...,2 \ell_{k}$, one has a hyperedge between $k$ black vertices with valency $\ell_{1},\ell_{2},...,\ell_{k}$, respectively. 
\end{itemize}
Note that even though the mobiles forming a hypermobile are plane trees, a hypermobile's hyperedges in general do not have associated planar embedding, so hypermobiles are not themselves maps, unlike mobiles.

We can now define the set of hypermobiles 
  $\mathbb{T}(S';v)$ with $v$-many labelled white vertices  constructed from components of $S'$.
  To summarize, the valency of a black vertex in such a mobile can be half the length of any boundary of any 2-cell in $S$ and the valency of any hyperedge can be the number of boundaries of any 2-cell in $S$. 

\subsection{Branch-pointed stuffed maps}
 As mentioned in the Introduction, one motivation to study stuffed maps is to establish convergence, in some sense, as a random discrete metric space to a random continuous metric space. All methods known to the author require a bijection to some path-connected tree-like object. However, one distinguishing feature of stuffed maps is that they may have distinct path-connected components connected by branches, which makes it not always possible to label vertices with their graph distances from the source vertex. To facilitate such bijections, we will slightly change the type of stuffed map studied in this work by decorating them and restricting their 2-cells. Define   a \textit{branch-pointed stuffed} (BPS) map as a stuffed map such that each boundary of a  multi-boundary 2-cells has a distinguished vertex called a \textit{spurious point}. There is an additional distinguished point that does not coincide with a spurious point.  The source cannot coincide with a spurious point. One can define the following set of BPS maps analogously to our definition of $\mathbb{M}_{\ell}(S;v)$ for stuffed maps: denote it $\overline{\mathbb{M}}_{\ell}(S;v)$.

 Working with rooted BPS maps allows one to give each vertex a positive integer label, which is key to the bijection.  Note that on the level of enumeration, computing the generating function of BPS maps suffices to enumerate stuffed maps and vice versa. For more details, see the Remark in Section 3.1.

\subsection{A bijection with hypermobiles}
In this section we will construct and prove the main bijection between planar BPS maps and hypermobiles. First, consider the mapping $\Phi$ that takes a rooted pointed BPS map $M \in\overline{\mathbb{M}}^{\circ}(S;v)$ to a hypergraph $\Phi(M)$, constructed as follows:
     \begin{enumerate}

\item Set the colour of all vertices of $M$ to be white. For every multi-boundary 2-cell, draw a hyperedge between the spurious points of each boundary. On the pointed component, assign the label of zero to the source and label every vertex belonging to the pointed component with its graph distance from the source.
\item Move every hyperedge to the white vertex in a clockwise direction along the boundary face of each spurious point.
         \item   In every 2-cell draw a black vertex associated with each boundary.  For every edge $e$ of $M$, let $v$ be the vertex with the largest label. Draw an edge from the vertex $v$ to the black vertex in the face to the left $e$ when directed at $v$. If $e$ is the rooted edge, then take the new edge to be rooted, oriented away from $v$. 
         \item Erase all original edges, all vertices that are the minimal label on a component, and all branches of $M$.  Uniformly shift all labels so that the source of the rooted edge is labelled zero.
     \end{enumerate}

   For an example of this construction, see Figure \ref{fig:bijection p1}.
     
\begin{remark}
    Note that for Step 3 in the construction of $\Phi(M)$, ``...left of $e$ when directed at v.'' refers to left in the local frame of reference on the surface, which may not be the same as left as seen in a figure's representation of the stuffed map. For example, in Figure 4 the largest component's left will be opposite the left of the other components as it is drawn on the page.   One can visualize different components on opposite sides of the surface but squished together to be succinctly represented in the figure.
\end{remark}

 \begin{figure}[p]\centering
\subfloat[]{\label{a}\includegraphics[width=.45\linewidth]{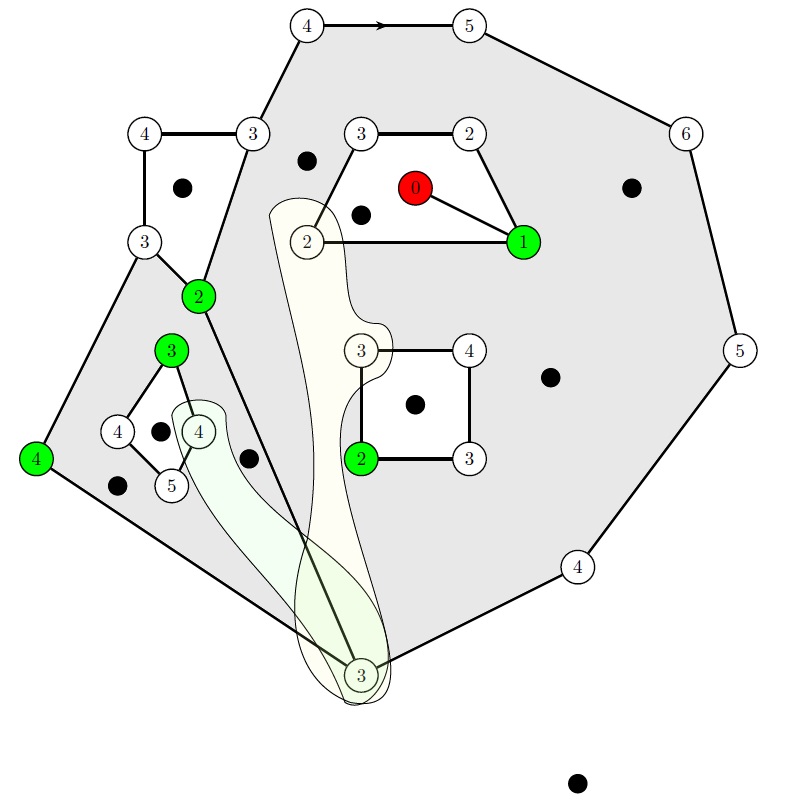}}\hfill
\subfloat[]{\label{b}\includegraphics[width=.45\linewidth]{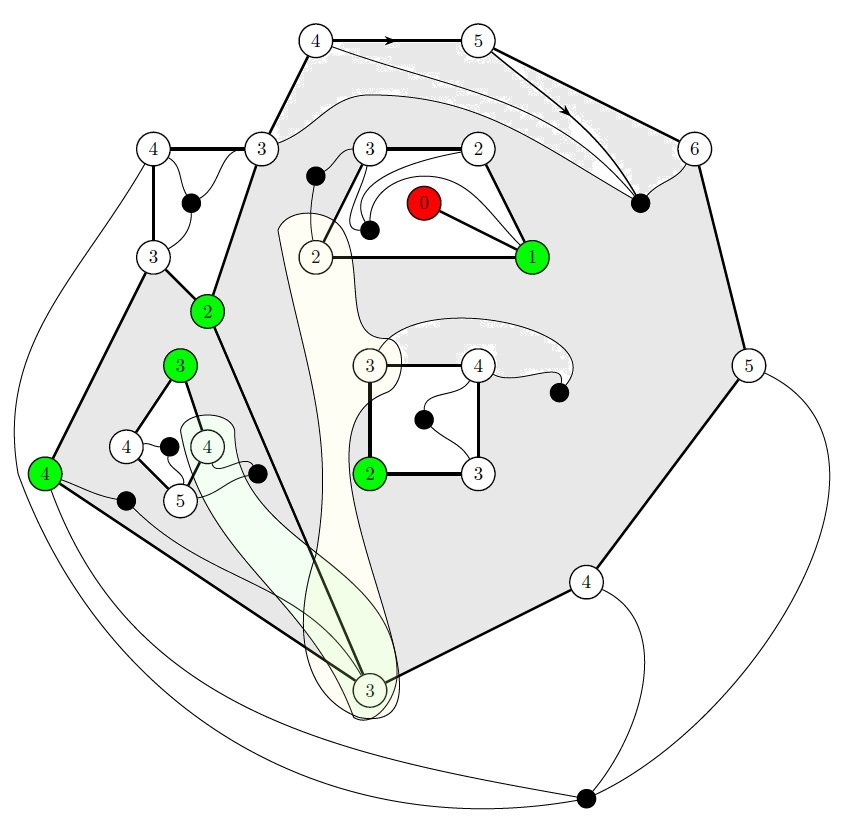}}\par 
\subfloat[]{\label{c}\includegraphics[width=.45\linewidth]{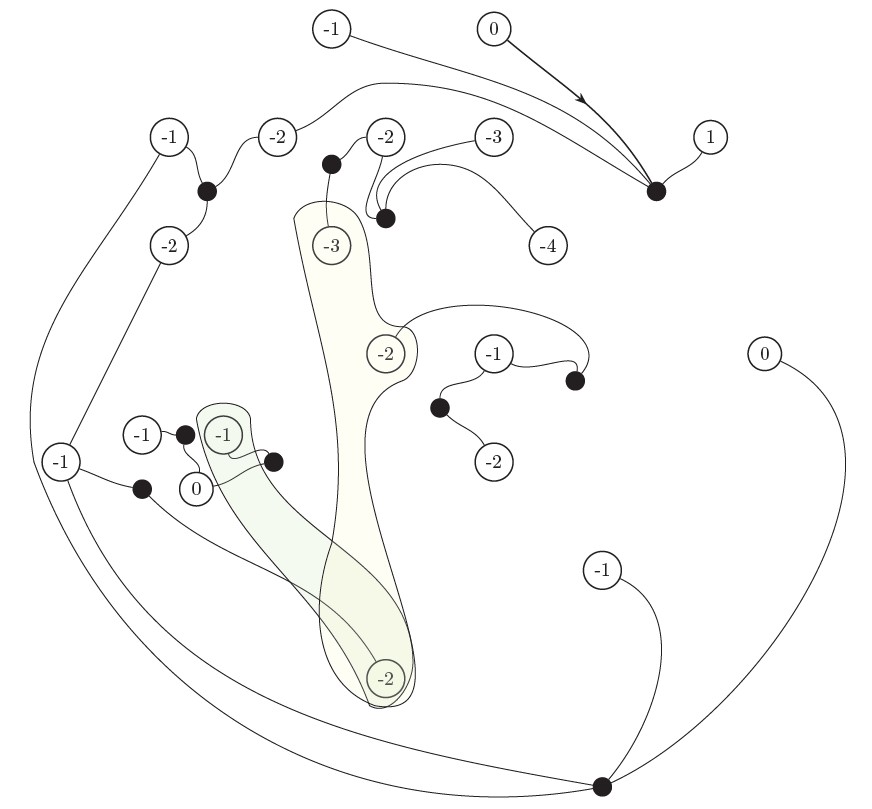}}
	\caption{Stages of the construction of a hypermobile from the  BPS map seen in Figure 1 (b). Sub-figure (a) is a BPS map whose vertices are labelled with the distance of the source and whose edges next to gate vertices are connected by hyperedges. In sub-figure (b), edges are drawn according to the same rules as the BDFG bijection on each component. Sub-figure (c) is the resulting hypermobile after removing all the original edges and branches, and shifting labels.}
\label{fig:bijection p1}
\end{figure}

Consider a mapping $\Psi$ that takes a hypermobile $T \in \mathbb{T}(S;v)$ to a map $\Psi(T)$, constructed as follows:
\begin{enumerate}
    \item For each mobile $m$ in $T$ add a white vertex labelled one less than the minimum label $n_{\text{min}}(m)$ of $m$.
        \item For every corner of a white vertex in $m$,  then connect it to the next vertex in a clockwise direction with $n-1$ in the contour. If this corner is the corner to the left of a root, the new edge is rooted pointing away for the corner. 
       \item Remove all the original edges and black vertices of $m$. Replace each hyperedge with a 2-cell whose boundaries correspond to the hyperedges' vertices. The minimum labelled vertex on the origin mobile is taken to be the origin. Spurious points are taken to be the vertex in a counterclockwise direction along each boundary face of a gate vertex.
   
\end{enumerate}
For an example, see Figure \ref{fig:bijection p2}.

   \begin{figure}[p]\centering
\subfloat[]{\label{a}\includegraphics[width=.45\linewidth]{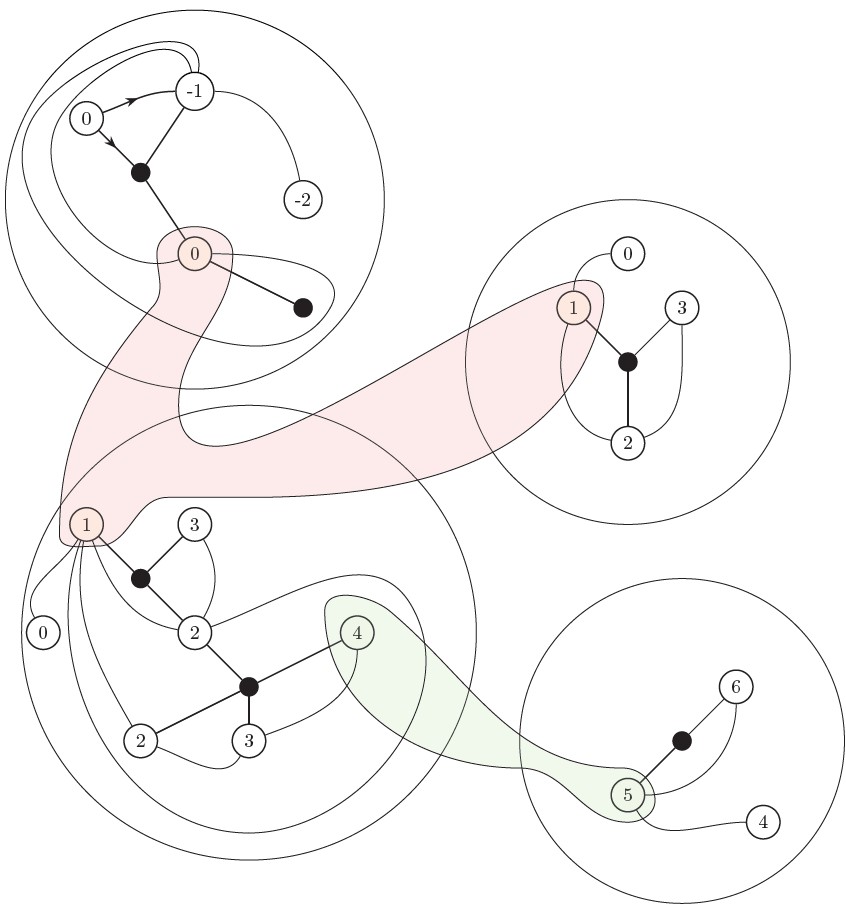}}\hfill
\subfloat[]{\label{b}\includegraphics[width=.45\linewidth]{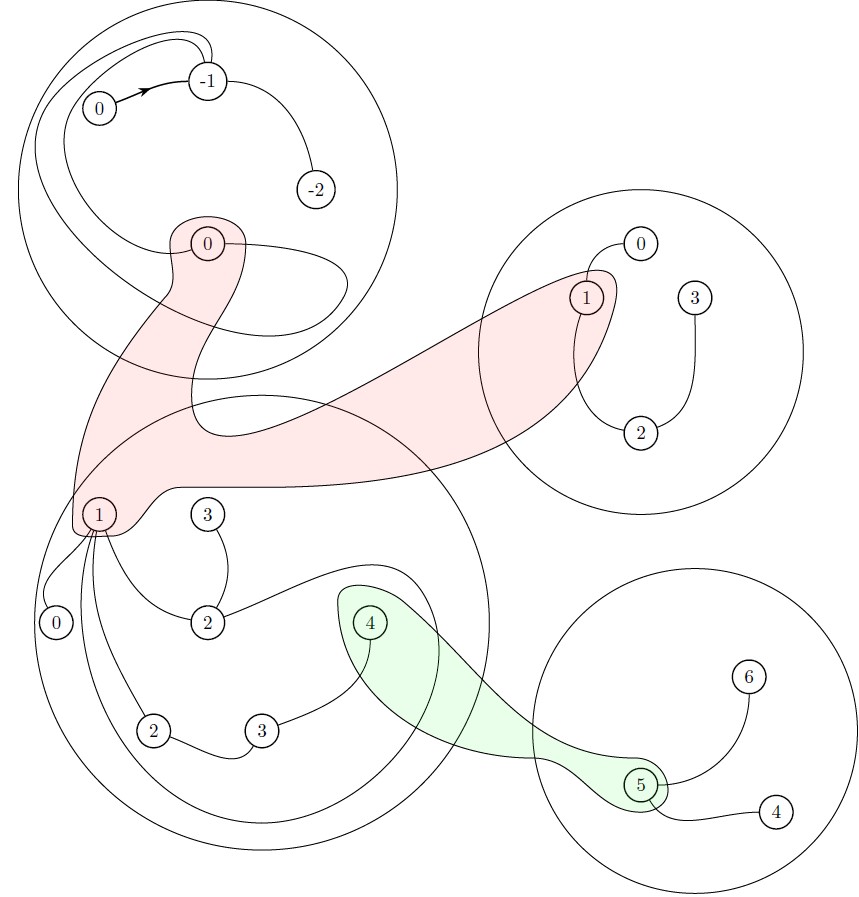}}\par 
\subfloat[]{\label{c}\includegraphics[width=.45\linewidth]{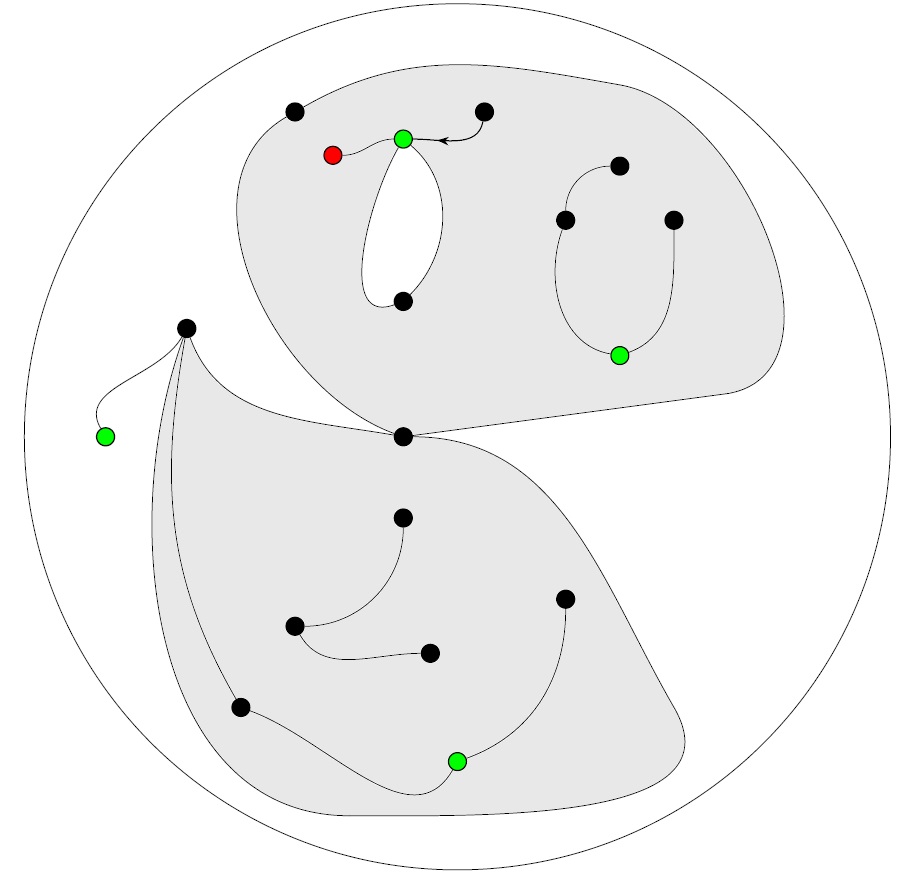}}
\caption{Stages of the construction of a BPS map from the hypermobile seen in Figure 2 (b). Sub-figure (a) and (b) have additional edges and vertices drawn according to the BDFG rules on each mobile. Sub-figure (c) is the resulting BPS map.}
	\label{fig:bijection p2}
\end{figure}

\begin{theorem}\label{thm: bijection}
    The mapping $\Phi$ is equal to  $\Psi^{-1}$ and is a bijection between rooted and pointed BPS maps $\overline{\mathbb{M}}^{\circ}(S;v)$ and  hypermobiles $\mathbb{T}(S;v)$.
\end{theorem}
\begin{proof}

 First, we will show that $\Phi$ maps into  $\mathbb{T}(S;v)$. Let $M \in \overline{\mathbb{M}}^{\circ}(S;v)$. It is clear that $\Phi(M)$ is a labelled hypergraph. On each  component one has applied the BDFG bijection, so  they correspond to mobiles. The resulting graph is a collection of mobiles connected by the occasional hyperedges. Consider a host graph $G$ formed from $\Phi(M)$ by replacing each hyperedge with a linear sequence of $n-1$ red edges connecting all black vertices in that hyperedge. If we can show that $G$ is a tree, then this suffices to show that $\Phi(M)$ is a hypertree.  Suppose for purposes of contradiction that $G$ contains a loop. No loop can be formed with just ordinary edges since this would contradict that one of the mobiles is a tree. Thus, such a loop must contain at least one red edge. However, any such sequence of edges that  forms a cycle would have to come from a sequence of branches and components that topologically forms a handle in $M$, which contradicts the planarity of $M$. Thus, $G$ has no cycles. 

What remains is to show that $G$ is connected.   By construction, recall that there are $F(M)$ many black vertices, and $V(M)-K(M)$ many white vertices. The number of edges between white and black vertices  is $E(M)$, since for each existing edge precisely one edge is created. The number of red edges is $\sum_{\ell}|\ell|$, where the sum is over all distinct linear sequences $\ell$ of red edges in $G$. Note that any such sequence corresponds to the number of boundaries of the corresponding branch in $M$ minus one, i.e. 
\begin{equation*}
    \sum_{\ell}|\ell| = \sum_{i=2}(i-1)B_{i}(M),
\end{equation*}
where $B_{i}(M)$ denotes the number of branches in $M$ with $i$ boundaries. Recall from Section \ref{sec: planar stuffed maps} the associated hypertree to $M$ and note that it must have a host tree. As a tree, this implies that by its construction $B(M) = K(M) -1$ and the degree sum formula can be written as $\sum_{i=2} iB_{i}(M) = 2B(M)$. Putting this all together:
\begin{align*}
    V(G) - E(G) &= (F(M) + V(M) - K(M)) -(E(M) +\sum_{i=2}(i-1)B_{i}(M))\\
    &= F(M) + V(M) - E(M) -K(M) -\sum_{i=2}(i-1)B_{i}(M))\\
    &= K(M) -\sum_{i=2}(i-1)B_{i}(M))\\
    &= K(M)+B(M) -\sum_{i=2}iB_{i}(M)) \\
    &= K(M) - B(M)\\
    &=1.
\end{align*}
Hence, $G$ is connected and therefore $\Phi(M)$ is a hypertree in $\mathbb{T}(S;v)$.

Second, we will show that $\Psi$ maps into  $\overline{\mathbb{M}}(S;v)$. It follows from the BDFG bijection that each mobile in $T$ is a bipartite path-connected ordinary map in the image of $\Psi$. The last step of the construction just adds branches between these ordinary maps making them components of a bipartite stuffed map. If $\Psi(T)$ was topologically disconnected, then this would imply that $T$ is disconnected. It remains to be shown that $\Psi (T)$ is planar. By construction the number of vertices of $\Psi (T)$ is equal to the number of white vertices of $T$ plus the number of mobiles $V_{w}(T) + m(T)$, the number of edges is the same, and the number of faces is the number of black vertices $V_{b}(T)$. Thus, one can write
\begin{align*}
    \chi (\Psi(T)) &= V(\Psi(T)) - E(\Psi(T)) + F(\Psi(T)) \\ 
    &= V_{w}(T) + m(T) - E(T) + V_{b}(T)\\
    & = 2 m(T).
\end{align*}
The number of path-connected components of $T$ is $m(T)$, and since $\chi (\Psi(T)) = 2 K(\Psi(T)) -2g(\Psi(T))$, we can conclude that $\Psi(T)$ is planar.

Lastly, it needs to be shown that $\Phi = \Psi^{-1}$. It follows from the BDFG bijection that components correspond to mobiles, and that roots and distinct white vertices will be preserved under $\Psi\circ \Phi$ and $\Phi\circ \Psi$. Since branches correspond to hyperedges directly in both constructions it is clear that their structure is maintained. Furthermore, in step 3 of the construction  of $\Phi(M)$, the choice of spurious point is inverse to the choice of gate vertices in step 2 of the construction of $\Psi(M)$.
\end{proof}

\section{Stuffed map enumeration}
\subsection{Generating functions}

With the goal of constructing generating functions for stuffed maps, to each element of $S$ with boundaries of lengths $\ell_{1},...,\ell_{k}$ we assign a weight $t_{\ell_{1},...,\ell_{k}}$ and we refer to the weights collectively as $\mathbf{t}$. Note that there is no order of indices for each weight, since 2-cells have no order of boundaries. We also assign each vertex a weight $t$ that is a redundant parameter, but is included for the convenience of computations with generating functions. We define the generating function of stuffed maps glued from elements of $S$ as 
\begin{equation}
	\mathcal{T}_{\ell}(t,\mathbf{t}) = \sum_{v=1}^{\infty} t^{v} \sum_{M\in  \mathbb{M}_{\ell}(S;v)} \prod_{k=1}^{d}\prod_{\ell_{1},...\ell_{k}= 1} \frac{t_{\ell_{1},...\ell_{k}}^{n_{\ell_{1},...\ell_{k}}(M)}}{|\text{Aut}(M)|},
\end{equation}
where $n_{\ell_{1},...\ell_{k}}(M)$ is the number of 2-cells with boundaries of lengths $\ell_{1},...\ell_{k}$ used to construct a given map $M$. Throughout this paper, the choice of set $S$ is clear from context. When the generating function only enumerates ordinary maps, we will denote it $T_{\ell}(t,\mathbf{t})$. In general, the various generating functions of the stuffed maps will be calligraphic versions of the generating functions of ordinary maps. Additionally, to avoid cumbersome expressions, we will drop the weights in the brackets in the generating function when they are not immediately relevant.

From generalizing the well-known Euler characteristic argument (see Theorem 1.2.1 of \cite{eynard2016counting} for example) we will show that $\mathbb{M}_{\ell}(S;v)$ is a finite set, implying that $\mathcal{T}_{\ell}(t,\mathbf{t})$ is a well-defined formal series. 
\begin{prop}\label{prop: finite}
       Let $S$ be a (not necessarily finite) set of 2-cells such that for any element the sum of the lengths of its boundaries is at least three.  Then the set of stuffed maps $\mathbb{M}_{\ell}(S,v)$ is finite and $\mathcal{T}_{\ell}$ is a well-defined formal power series in $t$ with polynomial coefficients, i.e. $\mathcal{T}_{\ell}\in \mathbb{Q}[\mathbf{t}][[t]]$.
\end{prop}
\begin{proof}
    Let $M \in \mathbb{M}_{\ell}(S;v)$. The Euler characteristic of $M$  is equal to 
    \begin{equation}\label{eq:Euler}
        2K(M) =  V(M) - E(M) +F(M),
    \end{equation}
    where $K(M)$ is the number of connected components of $M$, $V(M)$ is the number of vertices, $E(M)$ the number of edges, and $F(M)$ the number of faces. The number of faces is one boundary of length $\ell$ plus the boundaries of each 2-cell: 
\begin{equation}\label{eq:faces}
    F(M) = 1 +\sum_{k=1}^{d}k\,\sum_{\substack{\ell_{1},\ell_{2},...,\ell_{k}=1\\ \ell_{1}+\ell_{2}+...+\ell_{k}\geq 3}}  n_{\ell_{1},...,\ell_{k}}(M), 
\end{equation}
where $ n_{\ell_{1},...,\ell_{k}}(M)$ denotes the number of 2-cells with boundary lengths of $\ell_{1},...,\ell_{k}$ used to construct $M$. The number of edges in $M$ is half the sum of the number of half-edges used to construct it, i.e. half the number of edges of faces: 
\begin{equation}\label{eq:edges}
     E(M) = \frac{1}{2}\left(\ell + \sum_{k=1}^{d}k\,\sum_{\substack{\ell_{1},\ell_{2},...,\ell_{k}=1\\ \ell_{1}+\ell_{2}+...+\ell_{k}\geq 3}} (\ell_{1}+...+\ell_{k}) n_{\ell_{1},...,\ell_{k}}(M)\right).
\end{equation}
Applying equations \eqref{eq:faces} and \eqref{eq:edges} to \eqref{eq:Euler} we find that
    \begin{align*}
        V(M)-2K(M)+1-\frac{\ell}{2}& = \frac{1}{2}\left(\sum_{k=1}^{d}k\,\sum_{\substack{\ell_{1},\ell_{2},...,\ell_{k}=1\\ \ell_{1}+\ell_{2}+...+\ell_{k}\geq 3}} (\ell_{1}+...+\ell_{k}-2) n_{\ell_{1},...,\ell_{k}}(M)\right)\\
        & \geq \frac{1}{2}\left(\sum_{k=1}^{d}k\,\sum_{\ell_{1},\ell_{2},...,\ell_{k}=1}  n_{\ell_{1},...,\ell_{k}}(M)\right).\\
    \end{align*}

     By the definition of the set $\mathbb{M}_{\ell}(S,v)$ we have that the boundary length $\ell$ and the number of vertices $V(M) = v$ is fixed, and that $K(M) \geq 1$, hence the left-hand side of the above inequality is bounded. Hence, the number of faces used to glue $M$  is bounded, and thus the number of possible maps $M\in \mathbb{M}_{\ell}(S,v)$ is finite. This with the definition of $\mathcal{T}_{\ell}$ implies $\mathcal{T}_{\ell}\in \mathbb{Q}[\mathbf{t}][[t]]$.
\end{proof}

It is also worth noting that when $\ell$ is odd, then $\mathcal{T}_{\ell} =0$. This follows from the fact that no matter what even 2-cells belong to the set $S$, with an odd boundary to glue, there will always be an odd number of half-edges, and hence no such gluing exists.

\begin{remark}\label{remark: BPS}
    Let $\overline{\mathcal{T}}_{\ell}(t,\mathbf{t})$ denote the generating function above but for BPS maps.  The difference between the generating functions of stuffed maps and branch-rooted stuffed maps is a symmetry factor of $\ell_{1} \cdots \ell_{k}$ for every 2-cell with boundaries of lengths $\ell_{1},...,\ell_{k}$ and a factor of $t$  for each boundary, i.e.
\begin{align*}
\mathcal{T}_{\ell}\left(t,...,\frac{t_{\ell_{1},...\ell_{k}}}{t^{k}\ell_{1}\cdots\ell_{k}},...\right)&=  \sum_{v=1}^{\infty} t^{v} \sum_{M \in  \mathbb{M}_{\ell}(S;v)} \frac{1}{|\text{Aut}(M)|}\prod_{i=1} t_{i}^{n_{i}(M)}\prod_{k=2}\prod_{\ell_{1},...\ell_{k}= 1} \frac{t_{\ell_{1},...\ell_{k}}^{n_{\ell_{1},...\ell_{k}}(M)}}{t^{k} \ell_{1} \cdots \ell_{k}}\\
&= \overline{\mathcal{T}}_{\ell}(t,\mathbf{t}).
\end{align*}
This factor comes from the choice of root for each boundary of a branch. This means that since one can enumerate planar stuffed maps, one can equivalently enumerate branch-rooted stuffed maps via this rescaling of generating functions.  

Note that the automorphism group of any BPS map is trivial. This follows from the fact that every connected component of a BPS map will have at least one distinguished point. By the argument in Proposition \ref{prop: finite}, one can see that the generating functions of BPS maps have integer coefficients, that is, $\overline{\mathcal{T}}_{\ell}(t,\mathbf{t}) \in \mathbb{N}[\mathbf{t}][[t]]$.
    
\end{remark}

 \subsection{A review of ordinary  map enumeration}
Let's very briefly review the standard methods used to enumerate ordinary maps, whose techniques will be useful when enumerating stuffed maps. Tutte originally constructed a recursive bijection for ordinary maps by removing the rooted edge of a boundary and reassigning new root(s) \cite{tutte1968enumeration}. The resulting recursive formula for generating functions is 
\begin{equation}\label{eq:Tutte}
    T_{\ell} = \sum_{i=0}^{\ell -2}T_{\ell -2 +i} T_{i} + \sum_{j=3}t_{j} T_{\ell -2 + j}.
\end{equation}
 For a detailed review, we refer the reader to Chapter 2 of \cite{chapuy2014introduction}. Note that when all weights are zero, the recursion is for planar rooted one-faced maps, the solution of which is the Catalan numbers, i.e. the number of  unicellular maps.  By defining a generating function of $T_{\ell}$'s, one can then write the above equations as a quadratic equation. We define the formal generating function of maps with one boundary as
\begin{equation*}
    W(x) = \sum_{\ell=1}^{\infty}\frac{T_{\ell}}{x^{\ell+1}}.
\end{equation*}
Then by multiplying equation \eqref{eq:Tutte} by $\frac{1}{x^{\ell+1}}$ and summing we can write it as 
\begin{equation}
    0=W(x)^2 - V(x) W(x) + R(x),
\end{equation}
where 
\begin{equation*}
    V(x) = x - \sum_{i=3} t_{i} x^{i-1}
\end{equation*}
and 
\begin{equation*}
    R(x) = t - \sum_{q =0} x^{q} \sum_{i = q+2} t_{i} T_{i-q-2}.
\end{equation*}
The solution to this quadratic relation is of the form
\begin{equation*}
    W(x) = \frac{1}{2}\left( V(x) \pm \sqrt{V(x)^2 - 4 R(x)}\right).
\end{equation*}
By the requirement that $\left[\frac{1}{x}\right]W(x) = t$, the choice of sign in front of the square root can be deduced to be minus. 

Now assume that the number of 2-cells used to construct maps is finite, i.e. $|S|<\infty$, and that all 2-cells have boundaries of even length. Both $V$ and $R$ are polynomials in $x$, and by a theorem of Brown \cite{brown1965existence} the discriminant is a polynomial that factorizes as 
\begin{equation*}
 V(x)^2 - 4 R(x) = M(x)^2 (x^2-4\gamma^2),
\end{equation*}
where $\gamma(t,\mathbf{t})^2 \in \mathbb{Q}[\mathbf{t}][[t]]$ and $M(x)\in \mathbb{Q}[x,\mathbf{t}][[t]]$. This result  implies that 
\begin{equation*}
    W(x)=\frac{1}{2}\left( V(x) - M(x)\sqrt{x^2 - 4\gamma^2}\right).
\end{equation*}

If one can compute the polynomial $M(x)$ and formal series $\gamma$, then they can enumerate all maps in $\mathbb{M}_{\ell}(S,v)$ for any $\ell,v\geq 0$. In fact, these quantities can be explicitly determined in terms of weights. The idea goes roughly as follows. As a complex function, $W(x)$ is easily defined on $\mathbb{C}\setminus [-\gamma,\gamma]$. The Zhukovsky transform 
\begin{equation*}
    x(z) = \gamma\left(z+\frac{1}{z} \right)
\end{equation*}
conformally maps  $\mathbb{C}\setminus [-\gamma,\gamma]$ to the exterior of the unit disk. The inverse transform is of the form
\begin{equation*}
    z(x) = \frac{1}{2\gamma}\left(x +\sqrt{x^2 - 4\gamma^2}\right).
\end{equation*}
The generating function $W(x)$ is much more simple to solve for in Zhukovsky space. For details, we refer the reader to Chapter 3.1 of \cite{eynard2016counting}.

Using this transformation one can show that the formal series $\gamma$ is the root of a polynomial equation
\begin{equation*}
    \gamma^2 -t = \sum_{n=2}\frac{(2n -1)!}{n! (n-1)!}t_{2n} \gamma^{2n}
\end{equation*}
equivalent to what is known as the \textit{tree equation}. It is called such because it is often derived by bijective methods involving decorated trees \cite{cori1981planar,bouttier2004planar}. From this equation, many of the statistical properties and the critical behaviour of random maps can be studied \cite{bouttier2003geodesic,marckert2007invariance}. Consider a map enumerated by $T^{\circ}_{2}$. If one zips closed the 2-gon boundary, the result is a map with no boundary but one rooted edge. Hence $ T^{\circ}_{2}-t$ is the generating of rooted pointed maps, and via the above formula $2 \gamma^2$ is also. The generating function $\gamma^2$ enumerates pointed maps with a distinguished edge. Via the BDFG bijection, this is the generating function of mobiles.

Further, all generating functions $T_{2\ell}$ can be computed in terms of $\gamma$  by expanding $T_{2\ell} =- \text{Res}_{x \rightarrow \infty} W(x) x^{2\ell}dx$ in Zhukovsky coordinates. The result found in \cite{eynard2016counting} is  
\begin{equation*}
    T_{2\ell} = \gamma^{2 \ell+1} \sum_{i=1}\frac{(2i-1)(2\ell)! u_{2i-1}}{(\ell-j+1)!(\ell+j)!},
\end{equation*}
where 
$u_{1} = \frac{t}{\gamma}$ and 
\begin{equation*}
    u_{2j-1}=-\sum_{j= 2k-1}t_{2j-2k+2}\frac{(2j -2k+1)!}{j!(j-2k+1)!} \gamma^{2j -2k+1}
\end{equation*} 
for $j \geq 2$. 

\subsection{Gasket decomposition}
As derived by Borot in \cite{borot2014formal}, planar stuffed maps have a nested structure, which can be seen in the form of a functional relation which we will now briefly review. Recall the gasket  $M'$ of a rooted stuffed map $M$ is the  rooted ordinary map obtained from $M$ by removing all branches and only keeping the connected component with the root. The resulting map will have some faces that come from 2-cells with one boundary and some that come from multi-boundary 2-cells.  In the latter case, we call the faces cement faces. Since $M$ is planar, the cement faces must all be from distinct 2-cells, otherwise such a 2-cell would create a handle in $M$. The original stuffed map $M$ can be reconstructed from the gasket $M'$ by replacing all the cement faces with their original 2-cells. At the level of generating functions, this corresponds to replacing the weights of the cement faces with a weighted combination of $\mathcal{T}_{\ell}$.

Given a generating function $\mathcal{T}_{\ell}$, let $T_{\ell}$ be the corresponding generating function of maps glued only from polygons.  The above decomposition leads to the following functional relationship 
\begin{equation}\label{eq:functional}
	\mathcal{T}_{\ell}(t,\mathbf{t}) = T_{\ell} (t,\mathbf{C}),
\end{equation}
where $\mathbf{C}$ denotes the sequence of weights
\begin{equation}\label{eq:functional weights}
	C_{i}(t,\mathbf{t}) = t_{i} -\sum_{k =2} \frac{1}{(k-1)!}\sum_{i_{2},...,i_{k}}t_{i,i_{1},...,i_{k}} \prod_{j=2}^{k}\frac{\mathcal{T}_{i_{j}}(t,\mathbf{t}) }{i_{j}}.
\end{equation}
This functional equation will be of use in the following sections.

\subsection{Tutte's recursion for stuffed maps}
Planar stuffed maps with a boundary are known to satisfy a generalized recursive equation derived in a  similar manner to ordinary maps. Let $S$ be the set of all possible unmarked 2-cells, then 
    \begin{align*}
    \mathcal{T}_{\ell} = \sum^{\ell-2}_{i=0} \mathcal{T}_{\ell-2 +i}\mathcal{T}_{i} +\sum_{k=1} k\sum_{\substack{\ell_{1},\ell_{2},...,\ell_{k}=1\\ \ell_{1}+\ell_{2}+...+\ell_{k}\geq 3}}\frac{t_{\ell_{1},\ell_{2},..,\ell_{k}}}{\ell_{2},..,\ell_{k}} \mathcal{T}_{\ell-2 + \ell_{1}} \prod_{r=2}^{k} \mathcal{T}_{\ell_{r}}.  
\end{align*}
This recursive formula can be obtained by applying the functional relation \eqref{eq:functional} \eqref{eq:functional weights} to Tutte's recursion for ordinary maps \eqref{eq:Tutte}. See the beginning of Section 4.1 of \cite{borot2014formal} for a graphical derivation. 

Just as with ordinary maps one can define the generating function
\begin{equation*}
    \mathcal{W}(x) = \sum_{\ell=1}^{\infty}\frac{\mathcal{T}_{\ell}}{x^{\ell+1}},
\end{equation*}
and obtain a quadratic equation to solve
\begin{equation}
    0=\mathcal{W}(x)^2 - \mathcal{V}(x) \mathcal{W}(x) + \mathcal{R}(x),
\end{equation}
where 
\begin{equation*}
    \mathcal{V}(x) = x - \sum_{i=2} \left(t_{i} -\sum_{k =2} \frac{1}{(k-1)!}\sum_{i_{2},...,i_{k}}t_{i,i_{1},...,i_{k}} \prod_{j=2}^{k}\frac{\mathcal{T}_{i_{j}} }{i_{j}}\right) x^{j-1}
\end{equation*}
and 
\begin{equation*}
    \mathcal{R}(x) = t - \sum_{q =0} x^{q} \sum_{i = q+2} \left(t_{i-q-2} -\sum_{k =2} \frac{1}{(k-1)!}\sum_{i_{2},...,i_{k}}t_{i-q-2,i_{1},...,i_{k}} \prod_{j=2}^{k}\frac{\mathcal{T}_{i_{j}} }{i_{j}}\right)  .
\end{equation*}

Just as before the solution is of the form
\begin{equation*}
    \mathcal{W}(x) = \frac{1}{2}\left( \mathcal{V}(x) - \sqrt{\mathcal{V}(x)^2 - 4 \mathcal{R}(x)}\right).
\end{equation*}
As noted in \cite{borot2014formal}, the functional equation combined with the factorization of the discriminant in the case of ordinary maps implies the same type of factorization for stuffed maps. Once again, the Zhukovsky method can be used to solve for an explicit solution. However, the problem is more complicated to solve in practice due to the appearance of the generating functions $\mathcal{T}_{\ell}'s$ in $\mathcal{V}(x)$. For an example of such a solution, see \cite{hessam2023double}.

\subsection{The tree equation for stuffed maps with pointed gaskets}
In bijections between maps and mobiles, a distinguished vertex is required to label vertices with graph distances from said point.   Recall that a map with a distinguished vertex is referred to as pointed, and such a vertex does not receive a weight $t$. Given any generating series of stuffed maps the corresponding generating function of pointed maps is found by formally differentiating with respect to $t$. Using properties of the Zhukovsky transform, one can show the relationship between the generating functions of pointed ordinary maps and $\gamma$:
  \begin{equation}
      T^{\circ}_{2\ell} =  {2 \ell \choose \ell} \gamma^{2 \ell}.
  \end{equation}
As discussed earlier, $T^{\circ}_{2} = 2\gamma^{2}$, the generating function of pointed maps or mobiles. Thus, via the functional relationship, $2\gamma^{2}(t,\mathbf{C})$ is the generating function of pointed rooted stuffed maps whose gasket is pointed. If we rescale and look at BPS maps, via the bijection of Theorem \ref{thm: bijection}, we can say that $2\gamma^{2}(t,\mathbf{C})$ is the generating function of hypermobiles whose origin mobile is rooted.

We will now apply the tree equation for ordinary maps and the functional equation \ref{eq:functional pointed main} to show that $\gamma^2$ for stuffed maps is the root of an equation.
\begin{prop}[Tree equation]\label{prop:tree}
    The generating function $2\gamma^2$ satisfies the following equation:
\end{prop}

\begin{align*}
    \gamma^2 -t &=  \sum_{n=1}\frac{(2n -1)!}{n! (n-1)!}t_{2n} \gamma^{2n} \\
    &+  \sum_{n=1}\frac{(2n -1)!}{n! (n-1)!}\gamma^{2n}\sum_{k =1} \frac{1}{(k-1)!}\sum_{i_{2},...,i_{k}}t_{2n,i_{1},...,i_{k}} \prod_{p=2}^{k}\frac{\gamma^{i_{p} +1}}{i_{p}}  S_{p}
\end{align*}
where 
\begin{equation*}
   S_{j} :=  \sum^{d/2}_{j=1}\frac{(i_{p})! (2j-1)!}{(i_{p}/2+j)! (i_{p}/2 -j+1)!} \sum_{q= 2k-1}t_{2q-2k+2}\frac{(2q -2k+1)!}{q!(q-2k+1)!} \gamma^{2q -2k+1}.
\end{equation*}
\begin{proof}
   Start by using Borot's functional equation and substituting in $C_{i}$ to the tree equation for ordinary maps:
\begin{align*}
	&\gamma^{2} -t \\
    &= \sum_{n=1}\frac{(2n -1)!}{n! (n-1)!}  C_{2n}\gamma^{2n}\\
    &= \sum_{n=1}\frac{(2n -1)!}{n! (n-1)!}  \left( t_{2n} -\sum_{k =1} \frac{1}{(k-1)!}\sum_{i_{2},...,i_{k}}t_{2n,i_{1},...,i_{k}} \prod_{j=2}^{k}\frac{\mathcal{T}_{i_{j}} }{i_{j}}\right)\gamma^{2n}.
\end{align*}
Note that one may write each generating function for ordinary maps  strictly in terms of $\gamma$ via Corollary 3.1.2 of \cite{eynard2016counting} which states 
\begin{equation}\label{eq:moments in terms of gamma}
    T_{2\ell} =  \gamma^{2\ell +1} \sum^{d/2}_{j=1}\frac{(2 \ell)! (2j-1)!}{(\ell+j)! (\ell -j+1)!} u_{2j-1},
\end{equation}
where 
$u_{1} = \frac{t}{\gamma}$ and 
\begin{equation*}
    u_{2j-1}=-\sum_{j= 2k-1}t_{2j-2k+2}\frac{(2j -2k+1)!}{j!(j-2k+1)!} \gamma^{2j -2k+1}
\end{equation*} 
for $j \geq 2$. 

Combining the above formulae gives us our result.
\end{proof}
If the set of stuffed maps is glued from a finite set of 2-cells, then this equation is a polynomial and $\gamma$ is an algebraic. In its full generality, it is not entirely clear to the author how to bijectively interpret this equation. However, we will discuss this bijection in detail for a specific case in Section \ref{sec:unstab;e}. Note that by using the above proposition, Corollary 3.1.2 of \cite{eynard2016counting}, and Borot's functional equation, one can explicitly compute any generating function $\mathcal{T}_{\ell}$ in terms of the Boltzmann weights. This will be done for an example in Section \ref{sec:unstab;e}.


\subsection{Adding a point}
The relationship between the generating functions of pointed stuffed maps and $\gamma$ is much more complicated than in the case of ordinary maps. Graphically, this complication boils down to the fact that when distinguishing a point on a rooted stuffed map it is very possible that the point lies on a different component than the root. More precisely, one of the following situations occurs. First, if the vertex removed belongs to the gasket, then the result is a stuffed map whose gasket is pointed. 

The second scenario is that the vertex removed belongs to a connected component that is not the gasket. There exists a unique shortest path of branches between the gasket and the connected component with a distinguished vertex. Removing the 2-cell with weight $t_{i_{1},...,i_{k}}$ of the first bridge along this path results in every connected component that was connected to the multi-edge being topologically disconnected from the others. One can realize each disconnected component as its own stuffed map. The connected component containing the gasket now has a distinguished face which was the boundary of the bridge. On the other components the other boundaries of the removed 2-cell become boundaries of the new stuffed map once a root is assigned. For an edge of length $i_{j}$ there are precisely $i_{j}$ choices. This corresponds to the factor of $\frac{1}{i_{j}}$.

It will follow from the functional equation that the process described above of assigning a point is a bijection. Recall that on the level of generating functions, distinguishing a vertex corresponds to formally differentiating with respect to $t$. Formally differentiating both sides of \eqref{eq:functional}  with respect to $t$ one obtains
\begin{equation}\label{eq:functional pointed}
		\mathcal{T}_{2\ell}^{\circ}(t,\mathbf{t}) 
 =  T_{2\ell}^{\circ} (t,C(t,\mathbf{t})) + \left(\sum_{i}\frac{\partial }{\partial t_{i}}T_{2\ell} (t,\mathbf{t})\right)|_{\mathbf{t} = C(t,\mathbf{t})}\sum_{k =1} \frac{1}{(k-1)!}\left(\sum_{i_{2},...,i_{k}}t_{2\ell,i_{2},...,i_{k}} \prod_{j=2}^{k}\frac{\mathcal{T}^{\circ}_{i_{j}}(t,\mathbf{t}) }{i_{j}}\right).
\end{equation}
Each term corresponds to precisely one step in the point assigning process. On the right-hand side the first term is the generating function for stuffed maps with pointed gaskets. The second term is the product of generating series of gaskets with one face removed and some collection of generating functions of stuffed maps. The following proposition further reduces this equation to a relationship between pointed stuffed maps and stuffed maps with pointed gaskets. The proof, somewhat implicitly, relies heavily on properties of the Zhukovsky transform.
\begin{theorem}\label{prop: functional pointed}
    The following functional relationship holds between the generating function of pointed stuffed maps with boundaries and the generating function of stuffed maps with pointed gaskets:

    \begin{align}\label{eq:functional pointed main}
	&\mathcal{T}_{2\ell}^{\circ}
 =  {2 \ell \choose \ell}\gamma^2 \nonumber\\
 &+ \left(\sum_{k}\frac{\gamma^{2k-\ell}}{k} \sum_{\substack{0\leq q \leq k\\ 0 \leq j \leq \ell\\ q+j =\ell}}q { 2k \choose q+k}  {2\ell \choose j}\right)\sum_{k =1} \frac{1}{(k-1)!}\left(\sum_{i_{2},...,i_{k}}t_{2\ell,i_{2},...,i_{k}} \prod_{j=2}^{k}\frac{\mathcal{T}^{\circ}_{i_{j}}}{i_{j}} \right),
\end{align}
where the first sum  is over all indices of weights for polygons. 
\end{theorem}
\begin{proof}
Recall Lemma 3.1.4 from \cite{eynard2016counting}, which gives us simplified forms of the derivatives of the generating function $W(x)$ in Zhukovsky space
\begin{equation*}
    x'(z)\left(\frac{\partial}{\partial t} W((x(z)) \right)|_{x(z)}= \frac{1}{z}
\end{equation*}
and 
\begin{equation*}
   x'(z)\left( \frac{\partial}{\partial t_{2k}} W((x(z))\right)|_{x(z)} = - \frac{1}{2k}\frac{\partial}{\partial z} ((x(z)^{2k})_{-},
\end{equation*}
where the minus subscript denotes the negative part of the Laurent polynomial. Transforming the first equation one finds that 
\begin{align*}
    \frac{\partial}{\partial t} T_{2\ell} &= - \text{Res}_{x \rightarrow\infty} \frac{\partial }{\partial t} W(x) dx \\
    &=  - \text{Res}_{z \rightarrow\infty} \left(\frac{\partial }{\partial t} W(x(z))\right)x(z)^{2\ell} x'(z) dz\\
    &=  - \text{Res}_{z \rightarrow\infty} \left(\frac{1}{z}\right)\gamma^{2\ell} \sum_{j=0}^{2\ell} {2\ell \choose j} z^{2\ell-2j} dz \\
    &=   {2\ell \choose \ell}\gamma^{2\ell}.
\end{align*}

Transforming the second equation one finds that 
\begin{align*}
    \frac{\partial}{\partial t_{2k}} T_{2\ell} &= - \text{Res}_{x \rightarrow\infty} \frac{\partial }{\partial t_{2k}} W(x) dx \\
    &=  - \text{Res}_{z \rightarrow\infty} \left(\frac{\partial }{\partial t_{2k}} W(x(z))\right)x(z)^{2\ell} x'(z) dz\\
    &=  -\text{Res}_{z \rightarrow\infty} \left(\frac{\gamma^{4k}}{2k} \frac{\partial }{\partial z} \left(\sum_{q =0}^{2k} { 2k \choose q} z^{2k -2q}\right)_{-}\right)\left(\gamma^{2\ell} \sum_{j=0}^{2\ell} {2\ell \choose j} z^{2\ell-2j}  \right) dz \\
    &= -\frac{\gamma^{4k}}{2k} \text{Res}_{z \rightarrow\infty} \left( \frac{\partial }{\partial z} \sum_{q =0}^{k} { 2k \choose q+k} z^{ -2q}\right)\left(\gamma^{2\ell} \sum_{j=0}^{2\ell} {2\ell \choose j} z^{2\ell-2j}  \right) dz \\
    &= \frac{\gamma^{4k}}{k} \text{Res}_{z \rightarrow\infty} \left( \sum_{q =0}^{k}q { 2k \choose q+k} z^{ -2q-1}\right)\left(\gamma^{2\ell} \sum_{j=0}^{2\ell} {2\ell \choose j} z^{2\ell-2j}  \right)  dz \\
    &= \frac{\gamma^{4k-2\ell}}{k} \text{Res}_{z \rightarrow\infty}  \sum_{q =0}^{k}\sum_{j=0}^{2\ell}q { 2k \choose q+k}  {2\ell \choose j} z^{2\ell-2j-2q-1}  dz \\
     &=\frac{\gamma^{4k-2\ell}}{k} \sum_{\substack{0\leq q \leq k\\ 0 \leq j \leq \ell\\ q+j =\ell}}q { 2k \choose q+k}  {2\ell \choose j}.
\end{align*}
Applying these formulae to the functional equation \eqref{eq:functional} that relates ordinary maps to stuffed maps completes the proof.
\end{proof}

As a consequence of this proposition, one can deduce that the generating functions of pointed rooted stuffed maps are algebraic if the set of 2-cells is finite. When this is the case, one can solve a finite system of non-linear equations for the generating function of pointed rooted stuffed maps with coefficients in $\mathbb{Q}[\gamma]$ and the generating function $\gamma^2$ is itself algebraic since it is a solution of the polynomial equation in Proposition \ref{prop:tree}.


     \section{Example: Bridged quadrangulations} \label{sec:unstab;e}
Bridged quadrangulations appear in the computation of a quartic matrix model with a $(\tr H^2)^2$ interaction \cite{das1990new,klebanov1995non,hessam2022noncommutative,khalkhali2024coloured}. Let $S$ consist of a quadrangle and a 2-cell with two boundaries of length 2. We refer to the latter 2-cell as a $\textit{bridge}$ and the maps glued from $S$ (as well as some boundaries of fixed lengths) as \textit{bridged quadrangulations}.

We start by considering ordinary maps glued from quadrangulations and 2-gons. From the tree equation for ordinary maps it is known that that the generating function for mobiles corresponding to rooted and pointed quadrangulations and 2-gons is a root of the polynomial
 $$\gamma^2  =t+ t_{2} \gamma^2 + 3 t_{4} \gamma^{4}.$$ The only root that agrees with the small $t$ behaviour of the solution is 
    \begin{equation*}
        \gamma^2 = \frac{1}{6 t_{4}}\left(1-t_{2} + \sqrt{1-2t_2 + t_{2}^2 -12 t t_{4}} \right).
    \end{equation*}
From equation \eqref{eq:moments in terms of gamma}, it is know that 
\begin{equation*}
    T_{2}(t,t_{2},t_{4}) = \gamma^3 \left( \frac{t}{\gamma} + t_{4} \gamma^{3}\right) = t \gamma^2 + t_{4} \gamma^{6}.
\end{equation*}

We can now apply Borot's functional relation to consider the generating function of stuffed maps with a pointed and rooted gasket, as is done more generally in Proposition \ref{prop:tree}. In this case, one needs to substitute $t_{2}$ with $ t_{2} + \frac{1}{2}t_{2,2} \mathcal{T}_{2}$, and then set $t_{2} =0$, resulting in 
\begin{equation}\label{eq:tree quads}
    \gamma^2 -t =  \frac{t t_{2,2}}{2} \gamma^4 + 3 t_{4} \gamma^{4} + \frac{t_{2,2}t_{4} }{2} \gamma^{8}. 
\end{equation}
The resulting root that agrees with the small $t$ behaviour is

\begin{equation*}
  \gamma^2 =   \frac{1}{2} \sqrt{Q-\frac{R}{3 \sqrt[3]{2} t_{2,2} t_{4}}}-\frac{1}{2} \sqrt{\frac{4}{t_{2,2}  t_{4}\sqrt{Q-\frac{R}{3 \sqrt[3]{2} t_{2,2} t_{4}}}}+Q+\frac{R}{3 \sqrt[3]{2} t_{2,2} t_{4}}},
\end{equation*}

where 
\begin{align*}
    R &:= 2 (-t t_{2,2}-6 t_{4})^3-144 t t_{2,2} t_{4} (-t t_{2,2}-6 t_{4})-108 t_{2,2} t_{4}\\
    &+\sqrt{\left(2 (-t t_{2,2}-6 t_{4})^3-144 t t_{2,2} t_{4} (-t t_{2,2}-6 t_{4})-108 t_{2,2} t_{4}\right)^2-4 \left((-t t_{2,2}-6 t_{4})^2+24 t t_{2,2}
   t_{4}\right)^3}
\end{align*}
and 
\begin{equation*}
   Q:= -\frac{\sqrt[3]{2} \left(t^2 t_{2,2}^2+36 t t_{2,2} t_{4}+36 t_{4}^2\right)}{3 R t_{2,2} t_{4}}-\frac{-t t_{2,2}-6 t_{4}}{3 t_{2,2} t_{4}}-\frac{t t_{2,2}+6 t_{4}}{t_{2,2} t_{4}}.
\end{equation*}

With this one can then compute any generating function of stuffed maps with boundaries. To see this, first recall the functional equation \eqref{eq:functional} for $\ell =2$ and equation \eqref{eq:moments in terms of gamma}:
\begin{align*}
    \mathcal{T}_{2}(t,t_{22},t_{4}) &= T_{2}\left(t, 1-\frac{t_{2,2}}{2} \mathcal{T}_{2}(t,t_{22},t_{4}),t_4\right)\\
    &= t \gamma\left(t, 1-\frac{t_{2,2}}{2} \mathcal{T}_{2}(t,t_{22},t_{4}),t_4\right)^2 - t_{4}  \gamma\left(t, 1-\frac{t_{2,2}}{2} \mathcal{T}_{2}(t,t_{22},t_{4}),t_4\right)^6 \\
    &= 2 t \gamma^2 - t_{4}\gamma^{6}.
\end{align*}
Using this formula, the generating function for any other boundary length can be computed using \eqref{eq:moments in terms of gamma} and the functional equation. Note that for odd $\ell$ no such unstable quadrangulation exists. This follows quickly from the functional equation and the fact that there are no quadrangulations with odd boundaries.

Proposition \ref{prop:tree} allows us to solve for the generating function of rooted pointed maps explicitly in terms of maps with pointed and rooted gaskets:
$$ \mathcal{T}^{\circ}_{2} = 2 \gamma^2 + \frac{t_{2,2}}{2}\gamma^2 \mathcal{T}^{\circ}_{2},$$
which can be rearranged to 
\begin{equation*}
     \mathcal{T}^{\circ}_{2} = \frac{2 \gamma^2}{1 -\frac{t_{2,2}}{2}\gamma^2}
\end{equation*}
for $t_{2,2}\gamma^2 \not =2.$

\printbibliography

	\end{document}